\documentclass[11pt,a4paper]{article}

\usepackage{geometry}
\usepackage{setspace}

\usepackage{bbm}
\usepackage{graphics} 
\usepackage{epsfig} 
\usepackage{mathptmx} 
\usepackage{times} 
\usepackage{amsmath} 
\usepackage{amssymb}  
\usepackage{amsfonts}
\usepackage[centerlast]{caption}
\usepackage{graphicx}          

\usepackage{tabularx}
\usepackage{booktabs}

\usepackage{amssymb}
\newtheorem{theorem}{Theorem}
\newtheorem{proposition}{Proposition}
\newtheorem{definition}{Definition}
\newtheorem{lemma}{Lemma}
\newtheorem{corollary}{Corollary}

\newtheorem{remark}{Remark}
\newtheorem{example}{Example}

\newtheorem{proof}{Proof}

\DeclareMathOperator{\diag}{diag}

\usepackage{amsmath}

\begin{document}

\title{Controllability of Heterogeneous Multi-Agent Networks}


\author{Bin Zhao,\thanks{Bin Zhao and Long Wang are with Center for Systems and Control, College of Engineering, Peking University, Beijing,
100871, China e-mail: bigbin@pku.edu.cn, longwang@pku.edu.cn.}
        Michael Z. Q. Chen,\thanks{Michael Z. Q. Chen is with School of Automation, Nanjing University of Science and Technology, Nanjing, Jiangsu 210094, China e-mail: mzqchen@outlook.com.}
        Yongqiang Guan,\thanks{Yongqiang Guan is with Center for Complex Systems, School of Mechano-Electronic Engineering, Xidian University, Xi'an 710071, China
                  e-mail: guan-jq@163.com.}
        and Long Wang$^*$
\thanks{This work is supported by NSFC (61375120, 61374053, 61603288). (\emph{Corresponding author: Long Wang.})}
        }

\maketitle

\begin{abstract}
The existing results on controllability of multi-agents networks are mostly based on homogeneous nodes. This paper focuses on controllability of heterogeneous multi-agent networks, where the agents are modeled as two types. One type is that the agents are of the same high-order dynamics, and the interconnection topologies of the information flow in different orders are supposed to be different. It is proved that a heterogeneous-topology network is controllable if and only if the first-order information topology is leader-follower connected, and there exists a Laplacian matrix, which is a linear combination of the Laplacian matrices of each order information, whose corresponding topology is controllable. The other type is that the agents are of generic linear dynamics, and the dynamics are supposed to be heterogeneous. A necessary and sufficient condition for controllability of heterogeneous-dynamic networks is that each agent contains a controllable dynamic part, and the interconnection topology of the network is leader-follower connected. If some dynamics of the agents are not controllable, the controllability between the agents and the whole network is also studied by introducing the concept of eigenvector-uncontrollable. Different illustrative examples are provided to demonstrate the effectiveness of the theoretical results in this paper.

\textbf{keywords:} Heterogeneous multi-agent networks; Controllability; Feedback gain; Leader-follower connected.
\end{abstract}

\section{Introduction}

During the past few decades, control of networked systems has become a popular topic because of the broad applications of networks in areas such as cognitive control of brain networks \cite{Gu15}, network neuroscience \cite{Bassett17}, controllability of dynamical systems and neural networks \cite{Sontag97,Levin93}. As a kind of common networks, multi-agent networks (MANs) have attracted attentions from many researchers. The distributed coordination control of MANs has many practical applications, for example, flocking in biology groups, unmanned aerial vehicle cooperative formation, and attitude adjustment of spacecrafts, to name a few. Some basic and important issues were studied including consensus problem \cite{Zhang16,Su13,Jing,Zheng16}, formation control \cite{Xiao09}, controllability and stabilizability \cite{Guan13,Guan16,Lu16}, and observability \cite{Lu}, etc. In the following, this paper takes MANs as a representative to discuss controllability of networks.

Controllability of MANs was proposed by Tanner for the first time in \cite{Tanner04}, where a necessary and sufficient condition was presented using the Laplacian matrix and the corresponding eigenvalues. Afterwards, researches on multi-agent controllability divided into two directions, one of which is on the necessary and/or sufficient conditions for controllability, while the other is on the methods to achieve and maintain controllability. The investigations on controllability conditions varies in different models and topologies, e.g., Wang et al. studied controllability of MANs with high-order dynamic agents and generic linear dynamic agents in \cite{Wang09}, where they showed that controllability of these systems is congruously determined by the interconnection topology. Ji et al. proposed a basic necessary condition named leader-follower connected topologies for controllability \cite{Ji09}. Following these works, Zhao et al. generalized the results from the viewpoint of left eigenvectors of the graph Laplacians in \cite{Zhao16}. Rahmani et al. provided some necessary conditions for controllability utilizing the equitable partition of the interconnection topology \cite{Rahmani09}. In addition, other useful necessary and sufficient conditions were obtained for controllability of some specific graphs, e.g., paths, cycles and grid graphs \cite{Parlangeli12,Notarstefano13}, to name a few. Besides, controllability of dynamic-edge multi-agent systems was proposed in \cite{Wang}, where PBH-like conditions were provided. Except for investigating the conditions of controllability, several new problems and properties were proposed for controllability. For example, controllability can be improved and achieved via different methods, which lead to controllability improvement. Algorithms of selecting proper leaders and adjusting edge weights to improve controllability were provided in \cite{Zhao16}, and protocol design to achieve controllability was studied in \cite{Ji15}. In parallel, structural controllability was studied under various models in \cite{Hou16,Lin74,Zamani009,Lou12}, where some necessary and sufficient conditions were also established. Additionally, robustness of controllability and structural controllability were investigated subject to failure of agents and communication edges, see \cite{Fu16,Rahimian13}.

However, the existing results on controllability mainly focus on homogeneous dynamic agents. In practice, for MANs, it is sometimes difficult to ensure that each agent has exactly the same dynamic. Even if the agents share the same high-order dynamic, the interconnection topologies of the information flow in different order may possibly be different \cite{Ren08}. For these kinds of networks, we collectively call them heterogeneous multi-agent networks. The study of controllability of heterogeneous MANs is just in its early stage \cite{Guan16}. Especially, for second-order heterogeneous-topology networks, the investigations only focused on consensus problems \cite{Ren08,Qin12}, and to the best of our knowledge, controllability of high-order networks with heterogeneous topologies has not yet been discussed.

Motivated by the above analysis, this paper studies controllability for heterogeneous multi-agent networks. The main contributions of this paper are summarised as follows: For high-order agent networks, the investigation focuses on the relationship between controllability and the topologies of information flow in different order. For generic-linear agent networks, since the agents may have different dimensions, feedback of each agent is introduced for communication, and the controllability is studied via the feedback gains.
Necessary and sufficient graphic conditions are established for heterogeneous MANs to be controllable (including two situations where the networks with high-order dynamic agents and with different generic dynamic agents).

This paper is organized as follows: Section II introduces and proposes some basic concepts and mathematic tools for this paper. Main results on controllability of heterogeneous MANs are obtained in Section III. Numerical examples are provided in Section IV to illustrate the theoretical results. Conclusions are drawn in Section V.

$\mathbf{Notations:}$ The set of $n$-dimensional real vectors is denoted by $\mathbb{R}^n$ and the set of $m\times n$ real matrices is denoted by $\mathbb{R}^{m\times n}$. Matrix $\diag(a_1,a_2,\cdots,a_n)$ is the matrix with principal diagonals $a_1,a_2,\cdots,a_n$, where $a_i\in\mathbb{R}^{n_i\times n_i},~i=1,2,\cdots,n$. Denote $(0,\cdots,0)^T,(1,\cdots,1)^T\in\mathbb{R}^n$ as $0_n$ and $1_n$, respectively. Let $e_i(n)$ represent the $i$-th column of the identity matrix $I_n$, and $(n)$ is omitted without misunderstanding. Let $sp\{L_1,L_2,\cdots,L_m\}$ denote the matrix space $\{L|L=k_1L_1+k_2L_2+\cdots+k_mL_m,k_i\in\mathbb{R},i=1,2,\cdots,m\}$. $\emptyset$ represents the empty set and $\otimes$ represents the Kronecker product. $S/T$ represents the set of all the elements in $S$ but not in $T$.

\section{Preliminaries}
\subsection{Graph theory}

An undirected graph $\mathbb{G}=(\mathbb{V},\mathbb{E})$ consists of a vertex set $\mathbb{V}=\{v_1,v_2,\cdots,v_n\}$, and an edge set $\mathbb{E}\subseteq \mathbb{V}\times\mathbb{V}$. In graph $\mathbb{G}$, $e_{ij}\in\mathbb{E}$ if and only if $e_{ji}\in\mathbb{E}$, and $v_i,~v_j$ are said to be adjacent with each other. The neighbor set of $v_j$ is denoted by $N_j=\{v_i\in\mathbb{V}|(v_i,v_j)\in\mathbb{E}\}$. The adjacency matrix of $\mathbb{G}$ is $A(\mathbb{G})=(a_{ij})\in \mathbb{R}^{n\times n}$, where $a_{ij}> 0$ is the weight of edge $e_{ji}$ (as well as $e_{ij}$), and $a_{ij}=0$ if $(v_j,v_i)\notin\mathbb{E}$. The Laplacian matrix of $\mathbb{G}$ is $L(\mathbb{G})=D-A$, $D=\diag(d_1,d_2,\cdots,d_n)$ where $d_k=\sum\limits_{i=1,i\neq k}^n a_{ki},~k=1,2,\cdots,n$.
If $\mathbb{G}_1,\cdots,\mathbb{G}_m$ contain the same number of nodes, then, the union graph of $\mathbb{G}_1,\cdots,\mathbb{G}_m$ is the graph whose adjacency matrix is $A(\mathbb{G}_1)+\cdots+A(\mathbb{G}_m)$.

\subsection{Model formulation}

\begin{definition}\label{xicontrollable}
For a linear system $\dot{x}=Ax+Bu,~A\in\mathbb{R}^{n\times n}, B\in \mathbb{R}^{n\times m}$, we call it $\{\xi_1,\cdots,\xi_s\}$-uncontrollable if:\\
1) $\xi_1^T,\cdots,\xi_s^T$ are linearly independent left eigenvectors of $A$ satisfying $\xi_i^TB=0,~i=1,2,\cdots,s$, and\\
2) For any left eigenvector of $A$, denoted as  $\tilde{\xi},$ satisfying $\tilde{\xi}^TB=0$, it holds $\tilde{\xi}\in span\{\xi_1,\cdots,\xi_s\}$.\\
Specially, if the system is controllable, we also call it $\emptyset$-uncontrollable.
\end{definition}

If system $\dot{x}=Ax+Bu,~A\in\mathbb{R}^{n\times n}, B\in \mathbb{R}^{n\times m}$ is controllable, we say $(A,B)$ is controllable for convenience.

\begin{example}
Consider the linear system
\begin{equation*}
  \dot{x}=\underbrace{\left(
            \begin{array}{ccc}
              1 & 2 & 0 \\
              1 & 1 & 1 \\
              2 & 1 & 0 \\
            \end{array}
          \right)}\limits_A x+\underbrace{\left(\begin{array}{c}
                     0 \\
                     -1 \\
                     3
                   \end{array}\right)}\limits_B u,
\end{equation*}
which is not controllable. The left eigenvectors of $A$ corresponding to eigenvalue $3$ are $k(5,6,2)^T,~k\neq0$, which are orthogonal to $(0,-1,3)^T$. The other eigenvectors are not orthogonal to $(0,-1,3)^T$. This makes the system be $\{(5,6,2)^T\}$-uncontrollable.
\end{example}

\begin{proposition}
For a $\{\xi_1,\cdots,\xi_s\}$-uncontrollable linear system $\dot{x}=Ax+Bu$, if $\xi_{i_1}^T,\cdots,\xi_{i_r}^T$ share the same eigenvalue of $A$, $1\leq r\leq s$, then, for any linearly independent $\eta_1,\cdots,\eta_r$ satisfying $\eta_j\in span\{\xi_{i_1},\cdots,\xi_{i_r}\},$ $j=1,2,\cdots,r$, the system is also $\Xi$-controllable, where $\Xi=(\{\xi_1,\cdots,\xi_s\}/\{\xi_{i_1},\cdots,\xi_{i_r}\})\cup\{\eta_1,\cdots,\eta_r\}$.
\end{proposition}
\begin{proof}
Let $\Theta_\xi=(\xi_{i_1},\cdots,\xi_{i_r})^T$, $\Theta_\eta=(\eta_1,\cdots,\eta_r)^T$, one obtains that there exists an invertible $Q\in \mathbb{R}^{n\times n}$ such that $\Theta_\eta=Q\Theta_\xi$. Thus, it yields $\Theta_\eta B=Q\Theta_\xi B=0$ and $\Theta_\eta A=Q\Theta_\xi A=\lambda Q\Theta_\xi=\lambda\Theta_\eta$, which means $\eta_1,\cdots,\eta_r$ are also left eigenvectors of $A$ with the corresponding eigenvalue $\lambda$. Obviously, $\eta_j,j=1,2,\cdots,r$ are linearly independent to $\{\xi_1,\cdots,\xi_s\}/\{\xi_{i_1},\cdots,\xi_{i_r}\}$, therefore $span\{\xi_1,\cdots,\xi_s\}=span\{\xi_1,\cdots,\xi_s\}/\{\xi_{i_1},\cdots,\xi_{i_r}\}\cup\{\eta_1,\cdots,\eta_r\}$. According to Definition \ref{xicontrollable}, the system is $\Xi$-uncontrollable.
\end{proof}

For multi-agent networks, information communication is based on distributed protocols and the interconnection topologies are modeled as undirected graphs in this paper, i.e., if there is a connection between two agents, the interconnection couplings between them are the same. The leader agents are supposed to receive external input(s) while the follower agents only receive information from neighbor agents. A multi-agent network is said to be controllable if for any initial state and any target state, the target state can be reached within a finite time by proper external input(s). If an MAN is controllable, we also say that the topology (of the network) is controllable. Next we introduce the most important concept for the interconnection topologies of MANs in this paper, named ``leader-follower connected'', which is also a basic necessary condition for multi-agent controllability.
\begin{definition}\cite{Ji09}
An interconnection graph $\mathbb{G}$ is said to be leader-follower connected if for each connected component of $\mathbb{G}$, there exist at least one leader in the component.
\end{definition}

\section{Main results on preserving controllability}

This section discusses controllability of heterogeneous multi-agent networks. In the first subsection, agents in the network are modeled as the same high-order dynamic systems, but the interconnection topologies of the information flow in different order are allowed to be different. In the second and the third subsections, the agents are modeled as different generic-linear dynamic systems, whereas the information interaction is under the same dimension.

\subsection{Controllability of Heterogeneous-Topology MANs}

For a high-order MAN, the dynamic of each agent is modeled as $x_i=x_i^{(1)},\dot{x}_i^{(1)}=x_i^{(2)},\cdots,\dot{x}_i^{(m-1)}=x_i^{(m)}$, and $\dot{x}_i^{(m)}=u_i$, where $u_i$ is the control input and $x_i^l$ is said to be the $l$-th order information, $l=1,2,\cdots,m,~i=1,2,\cdots,n$. The control inputs are supposed to follow the consensus-based protocol: $u_i= \sum\limits_{l = 1}^m {\sum\limits_{j \in N_i } {k_i a_{ij}^{(l)} (x_j^{(l)}  - x_i^{(l)} )} }  + u_{oi}$, where $a_{ij}^l$ is the interconnection coupling of the $l$-th order information between $x_i^l$ and $x_j^l$, $k_i$ is the feedback gain, $u_{oi}\in\mathbb{R}$ is the external control on the leader agent $v_i$, and $u_{oi}=0$ when $v_i$ is a follower. Then, the state of each agent is
\begin{equation}\label{1234}
\begin{aligned}
\left( {\begin{array}{*{20}c}
   {\dot x_i^{(1)} }  \\
   {\dot x_i^{(2)} }  \\
    \vdots   \\
   {\dot x_i^{(m)} }  \\
\end{array}} \right) =& \left( {\begin{array}{*{20}c}
   0_{m-1} & {I_{m - 1} }  \\
   0 & 0_{m-1}^T  \\
\end{array}} \right)\left( {\begin{array}{*{20}c}
   {x_i^{(1)} }  \\
   {x_i^{(2)} }  \\
    \vdots   \\
   {x_i^{(m)} }  \\
\end{array}} \right)                                                                \\
&+ \left( {\begin{array}{*{20}c}
   0_{m-1}  \\
   1  \\
\end{array}} \right)\left( {\sum\limits_{l = 1}^m {\sum\limits_{j \in N_i } {k_i a_{ij}^{(l)} (x_j^{(l)}  - x_i^{(l)} )} }  + u_{oi} } \right).
\end{aligned}
\end{equation}

Let $x=(x^{(1)T},\cdots,x^{(m)T})^T$, where $x^{(i)}=(x_1^{(i)},\cdots,x_n^{(i)})^T$. Then, summarize the compact form of the high-order MAN as

\begin{equation}\label{highcompact}
\begin{aligned}
\dot x =& \left(  \left( {\begin{array}{*{20}c}
   0_{m-1} & I_{m-1}  \\
   0 & 0_{m-1}^T  \\
\end{array}} \right) \otimes I_n - e_{m}\otimes e_1^{T}\otimes k_{1}L_1 \right.\\
&\left. -  \cdots  - e_m\otimes e_m^{T}\otimes k_{m}L_m \right)x + e_m\otimes Bu_o,
\end{aligned}
\end{equation}
where $e_i$ is the $i$-th column of $I_m$, $B=(e_{p+1},\cdots,e_{n})\in\mathbb{R}^{n\times (n-p)}$, $u_o\in\mathbb{R}^{n-p}$.

Network (\ref{highcompact}) is said to be controllable if for any initial state $x(t_0)=x_0$ and target state $x^*$, there exist a finite time $T>t_0$, a control input $u_o$, and $k_1,\cdots,k_m\in\mathbb{R}$ such that $x(T)=x^*$. We conclude the controllability of Network (\ref{highcompact}) as follows.

\begin{theorem}\label{high}
Network (\ref{highcompact}) is controllable if and only if there exists an $\tilde{L}\in sp\{L_1,L_2,\cdots,L_m\}$ such that $(\tilde{L},B)$ is controllable and the topology corresponding to $L_1$ is leader-follower connected.
\end{theorem}
\begin{proof}
Let $Q=  \left( {\begin{array}{*{20}c}
   0_{m-1} & I_{m-1}  \\
   0 & 0_{m-1}^T  \\
\end{array}} \right) \otimes I_n - e_{m}\otimes e_1^{T}\otimes k_{1}L_1 $ $-  \cdots  - e_m\otimes e_m^{T}\otimes k_{m}L_m$.
Network (\ref{highcompact}) is controllable if and only if there exist $k_1,k_2,\cdots,k_m\neq0$ such that $(Q,e_m\otimes B)$ is controllable, which is to say, no left eigenvector of $Q$ is orthogonal to $e_m\otimes B$. Suppose that $Q$ has an eigenvalue $\lambda$ with the corresponding left eigenvector $\beta^T=[\beta_1^T,\beta_2^T,\cdots,\beta_m^T]$, where $\beta_i\in \mathbb{R}^n, i=1,2,\cdots,m$. Then one has
\begin{equation}\label{high1}
\left\{
\begin{array}{l}
\begin{aligned}
    - k_1 \beta _m^T L_1  &= \lambda \beta _1^T   \\
   \beta _1^T  - k_2 \beta _m^T L_2  &= \lambda \beta _2^T   \\
    &\vdots   \\
   \beta _{m - 1}^T  - k_2 \beta _m^T L_m & = \lambda \beta _m^T .
\end{aligned}
\end{array}
\right.
\end{equation}
With simple transformation, one obtains
\begin{equation}\label{high2}
\left\{ {\begin{array}{l}
\begin{aligned}
   \beta _{m - 1}^T  &= k_m \beta _m^T L_m  + \lambda \beta _m^T  \\
   \beta _{m - 2}^T  &= k_{m - 1} \beta _m^T L_{m - 1}  + \lambda k_m \beta _m^T L_m  + \lambda ^2 \beta _m^T   \\
    &\vdots   \\
   \beta _{1}^T  &= k_2 \beta _m^T L_2  + \lambda k_3 \beta _m^T L_3  +  \cdots  + \lambda ^{m - 2} k_{m } \beta _m^T L_{m }  + \lambda ^{m - 1} \beta _m^T  .
\end{aligned}
\end{array}} \right.
\end{equation}
Substituting (\ref{high2}) into the first equation in (\ref{high1}) yields
\begin{equation}\label{high3}
\begin{array}{l}
\begin{aligned}
  - \lambda ^m \beta _m^T  &= k_1 \beta _m^T L_1  + \cdots + \lambda ^{m - 2} k_{m - 1} \beta _m^T L_{m - 1}  + \lambda ^{m - 1} k_m \beta _m^T L_m  \\
 & = \beta _m^T (k_1 L_1  + \lambda k_2 L_2  +  \cdots  + \lambda ^{m - 2} k_{m - 1} L_{m - 1}  + \lambda ^{m - 1} k_m L_m ).
\end{aligned}
 \end{array}
\end{equation}
According to the PBH Test, Network (\ref{highcompact}) is controllable if and only if for all $\lambda\in\Lambda(Q)$, the corresponding $\beta_m(\lambda)\neq0$ satisfies $\beta_m^TB\neq0$. On one hand, if Network (\ref{highcompact}) is controllable, it means $\beta_m^T$ is a left eigenvector of $\tilde{L}=k_1 L_1  + \lambda k_2 L_2  +  \cdots  + \lambda ^{m - 2} k_{m - 1} L_{m - 1}  + \lambda ^{m - 1} k_m L_m $ for some $\lambda$ and $k_1,k_2,\cdots,k_m$ not all being $0$. Obviously $\tilde{L}\in sp\{L_1,L_2,\cdots,L_m\}$ and $(\tilde{L},B)$ is controllable, especially, when $\lambda=0$, the left eigenvectors of $L_1$ corresponding to the eigenvalue $0$ are not orthogonal to $B$, i.e., the topology corresponding to $L_1$ is leader-follower connected. On the other hand, if there exist $k_1,k_2,\cdots,k_m$ not all being $0$ such that $(\tilde{L},B)$ is controllable, i.e., (\ref{high3}) holds for $\lambda=1$ (here $\lambda=1$ is not required to be an eigenvalue of $Q$), then, for almost all {\footnote {The choices of $k_1,\cdots,k_m$ have Lebesgue $1$ in $\mathbb{R}^m$.}} $k_1,\cdots,k_m\in\mathbb{R}$, $(k_1L_1+\cdots+k_mL_m,B)$ are controllable. Therefore, there exist $k_1,k_2,\cdots,k_m\neq0$ such that for finite choices of $\lambda\neq0$, $(\tilde{L},B)$ are all controllable, where $\tilde{L}=k_1 L_1  + \lambda k_2 L_2  +  \cdots  + \lambda ^{m - 2} k_{m - 1} L_{m - 1}  + \lambda ^{m - 1} k_m L_m $. Meanwhile, when $\lambda=0$, since the left eigenvectors of $L_1$ corresponding to the eigenvalue $0$ are not orthogonal to $B$, it means all the left eigenvectors of $\tilde{L}$, denoted one of whom as $\beta_m^T(\lambda)$, satisfy $\beta_m^T(\lambda)B\neq0$. Since $\beta^T$ can be calculated by (\ref{high3}), it also holds that $\beta^Te_m\otimes B\neq0$, i.e., Network (\ref{highcompact}) is controllable.
\end{proof}

\begin{remark}
The necessary and sufficient condition for Network (\ref{highcompact}) to be controllable has two parts, one is the existence of $\tilde{L}$ making $(\tilde{L},B)$ controllable, and the other is the topology corresponding to $L_1$ being leader-follower connected. The second condition is not contained in the first one, and is therefore important. As one can see, although the feedback gains $k_1,\cdots,k_n$ affect the eigenvalues of $Q$, they do not affect the $0$ eigenvalue. When $\lambda=0$, the only condition regarding the topologies on controllability falls on the (left) eigenvectors of the eigenvalue $0$ of $L_1$. Physically, the request of the second condition originates from the model $\dot{x}_i^{(m)}=u_i$, in which $x_i^{(1)}$ is the information of the lowest order. Corresponding to the eigenvalue $0$, only the information of the lowest order affects the controllability of the network. If some of the agents cannot receive the lowest-order information from any leader, the lowest-order states of the agents can not be completely controlled, which makes the network uncontrollable. However, for the non-zero eigenvalues, the information in different orders can jointly affect the states of each order, and one can make the network controllable by adjusting proper feedback gains.
\end{remark}

\begin{corollary}\label{highcoro}
\cite{Wang09} For Network (\ref{highcompact}), if $L_1=L_2=\cdots=L_m\triangleq L$, then, the network is controllable if and only if $(L,B)$ is controllable.
\end{corollary}
\begin{proof}
For the trivial case of $L_1=L_2=\cdots=L_m\triangleq L$, one obtains $sp \{L_1,L_2,\cdots,L_m\}=\{kL|k\in\mathbb{R}\}$. On one hand, if $(L,B)$ is controllable, the left eigenvectors of $L$ are not orthogonal to $B$. Since $L\in\{kL|k\in\mathbb{R}\}$, it can be declared that network (\ref{highcompact}) is controllable according to Theorem \ref{high}. On the other hand, if the network is controllable, there exists a $k\neq0$ such that when $k_1=k_2=\cdots=k_m=k$, $(k(1+\lambda+\cdots+\lambda^{m-1})L,B)$ is controllable, i.e., $(L,B)$ is controllable.
\end{proof}

Theorem \ref{high} implies that the network can be controllable even if none of the interconnection topologies of each order information is controllable. This is because the information in different orders may uniformly contribute to achieving controllability. Corollary \ref{highcoro} is a main result in \cite{Wang09}. However, Theorem \ref{high} generalises it to the heterogeneous-topology situation. To verify the controllability of Network (\ref{highcompact}), we provide a method via the union graph.
\begin{proposition}
1. Network (\ref{highcompact}) is controllable if the union graph of $\mathbb{G}_1,\cdots,\mathbb{G}_m$ is controllable and $\mathbb{G}_1$ is leader-follower connected.\\
2. Network (\ref{highcompact}) is uncontrollable if the union graph of $\mathbb{G}_1,\cdots,\mathbb{G}_m$ is not leader-follower connected, or $\mathbb{G}_1$ is not leader-follower connected.
\end{proposition}
\begin{proof}
1. If the union graph of $\mathbb{G}_1,\cdots,\mathbb{G}_m$ is controllable, it means $(\tilde{L},B)$ is controllable, where $\tilde{L}=L_1+\cdots+L_m$. Since $\mathbb{G}_1$ is leader-follower connected, by Theorem \ref{high}, Network (\ref{highcompact}) is controllable.\\
2. If the union graph of $\mathbb{G}_1,\cdots,\mathbb{G}_m$ is not leader-follower connected, for all $k_1,\cdots,k_m\in\mathbb{R}$, the corresponding graph of $\tilde{L}=k_1L_1+\cdots+k_mL_m$ is not leader-follower connected, i.e., $(\tilde{L},B)$ is uncontrollable, which makes Network (\ref{highcompact}) uncontrollable. If $\mathbb{G}_1$ is not leader-follower connected, by Theorem \ref{high}, Network (\ref{highcompact}) is not controllable.
\end{proof}

\subsection{Controllability of Heterogeneous-Dynamic MASs}

The heterogeneous-dynamic MANs are modeled as
\begin{equation*}
  \dot{x}_i=A_ix_i+b_iu_i,~i=1,2,\cdots,n,
\end{equation*}
where $x_i\in\mathbb{R}^{m_i},A_i\in\mathbb{R}^{m_i\times m_i},b_i\in\mathbb{R}^{m_i}$. For simplicity, suppose that $(A_i,b_i),~i=1,2,\cdots,n$ are controllable pairs (the uncontrollable situation will be discussed later), i.e., there exist invertible matrices $T_1,\cdots,T_n$ such that

\begin{equation}\label{hetermodel}
T_{i}\left( {\begin{array}{*{20}c}
   {\dot x_{i1} }  \\
   {\dot x_{i2} }  \\
    \vdots   \\
   {\dot x_{im_i} }  \\
\end{array}} \right) = \left( {\begin{array}{*{20}c}
   0_{m-1} & {I_{m_{i} - 1} }  \\
   \alpha_{i1} & \tilde{\alpha}_i  \\
\end{array}} \right)T_{i}\left( {\begin{array}{*{20}c}
   {x_{i1} }  \\
   {x_{i2} }  \\
    \vdots   \\
   {x_{im_i} }  \\
\end{array}} \right) + \left( {\begin{array}{*{20}c}
   0_{m-1}  \\
   1  \\
\end{array}} \right)u_{i},
\end{equation}
where $ \left( {\begin{array}{*{20}c}
   0_{m-1} & {I_{m_{i} - 1} }  \\
   \alpha_{i1} & \tilde{\alpha_i}  \\
\end{array}} \right)=T_iA_iT_i^{-1}$ and $ \left( {\begin{array}{*{20}c}
   0_{m-1}  \\
   1  \\
\end{array}} \right)=T_ib_i$. As one can see, when $i\neq j$, $x_i$ and $x_j$ may have different dimensions. To ensure successful interaction between agents, different feedback gains are required for each agent to unify the dimensions of information. Thus, the distributed protocol is designed as $u_i= -\alpha_i^Tx_i  +  \sum\limits_{j \in N_i } { a_{ij} (\beta_j^T T_jx_j  - \beta_i^T T_ix_i )}   + u_{oi}$, where $\alpha_i=(\alpha_{i1},\tilde{\alpha}_{i}^T)^T,\beta_i\in\mathbb{R}^{m_i}$, $i=1,2,\cdots,n$.

The compact form of Network (\ref{hetermodel}) under the protocol is
\begin{equation}\label{hetercompact}
\tilde x = (\Omega-\tilde{L})\tilde x + (e_{\rho_1},\cdots,e_{\rho_r} )u_o,
\end{equation}
where $\rho_1=\sum\limits_{i=1}^{i_1}m_i,\cdots,\rho_r=\sum\limits_{i=1}^{i_r}m_i$, $\Omega=(\omega_1^{T},\omega_2^T,\cdots,\omega_n^T)^T,\tilde{L}=(\tilde{L}_1^{T},\tilde{L}_2^T,\cdots,\tilde{L}_n^T)^T,~\omega_i= e_i^T\otimes$ \\ ${\left( {\begin{array}{*{20}c}
    0_{m-1} & {I_{m_i - 1} }  \\
    0 & 0_{m-1}^T  \\
 \end{array}} \right)},~\tilde{L}_i=e_{m_i}\otimes (l_{i1},\cdots,l_{in})\otimes \beta_i^T,~i=1,2,\cdots,n$.

Network (\ref{hetercompact}) is said to be controllable if for any initial state $x(t_0)=x_0$ and target state $x^*$, there exist a finite time $T>t_0$, a control input $u_o$, and $\beta_i\in\mathbb{R}^{m_i},~i=1,2,\cdots,n$ such that $x(T)=x^*$. To investigate controllability of network (\ref{hetercompact}), the simplest single-leader case is firstly considered. The Laplacian matrix of the interconnection topology is denoted to be $L$ in the following.

\begin{lemma}\label{singleK}
For an MAN, if there is only one leader (i.e., $r=1$) and the interconnection topology is connected, then for any selection of the single leader, there exists a $K=\diag(k_1,\cdots,k_n),~k_i\in\mathbb{R},~i=1,2,\cdots,n$ such that $(LK,B)$ is controllable.
\end{lemma}
\begin{proof}
According to the PBH Test, $(LK,B)$ is controllable if and only if for any $\mu\in\mathbb{C}$, $rank(\mu I-LK,B)=n$. Since the interconnection topology is connected, the Laplacian matrix of the topology contains only a single $0$ eigenvalue. Suppose that $LT=T\Lambda=T{\left( {\begin{array}{*{20}c}
    0 & 0_{m-1}^T   \\
    0_{m-1} & \bar{\Lambda}  \\
\end{array}} \right)}$, where $\bar{\Lambda}=diag\{\lambda_2,\cdots,$ $\lambda_n\}$, $T=(\xi_1,\cdots,\xi_n)$. Denote $0=\lambda_1$, obviously, $\lambda_2,\cdots,\lambda_n$ are all positive real numbers, and $\xi_1,\cdots,\xi_n$ are the (left) eigenvectors corresponding to $0=\lambda_1<\lambda_2\leq\lambda_3\leq\cdots\leq\lambda_n$ satisfying $\xi_i^T\xi_i=1$ and $\xi_i^T\xi_j=0$ if $i\neq j$, which makes $\xi_1=\frac{1}{\sqrt{n}}1_n$. Considering that $(\mu I-LK,B)=(\mu TT^{-1}- T\Lambda T^{-1}K,TT^{-1}B)=T(\mu T^{-1}- \Lambda T^{-1}K,T^{-1}B)$, one obtains $rank(\mu I-LK,B)=rank(\mu T^{-1}- \Lambda T^{-1}K,T^{-1}B)$, where $T^{-1}=T^T$. \\
Denote $Q=(\mu T^{-1}- \Lambda T^{-1}K,T^{-1}B)$, since $B$ contains only one column, it follows that $q_{ij}=\xi_{ij}(\mu-\lambda_ik_j)$ if $j\leq n$ and $q_{i,n+1}=\xi_{il}$ (here $\xi_{ij}$ represents the $j$-th element of $\xi_i$). Next we prove that there exist $k_1,k_2,\cdots,k_n$ such that for any $\mu\in\mathbb{C}$, $Q$ is of full row rank. Denote $\bar{Q}=(\mu(\xi_2,\cdots,\xi_n)^T-\bar{\Lambda}(\xi_2,\cdots,\xi_n)^TK)$, if there exist $i,\cdots,i+s$ such that $\lambda_{i}=\cdots=\lambda_{i+s}=\tilde{\lambda}$, then, for any selection of $k_1\neq k_2\neq\cdots\neq k_n\neq0$, one concludes that the $i$-th, $\cdots,$ $(i+s)$-th rows of $\bar{Q}$ are linearly independent for all $\mu$. Otherwise, obtain some $(s+1)$ linearly independent columns of $(\xi_i,\cdots,\xi_{i+s})^T$, denoted as $M=(\eta_1,\cdots,\eta_{s+1})$, and $M=(m_{hj})_{(s+1)\times(s+1)}$ is invertible. Consider $M_\mu=(m_{hj}(\mu-\tilde{\lambda}k_{i+j-1}))_{(s+1)\times(s+1)}$, $M_\mu$ is not of full rank if and only if $\mu=\tilde{\lambda}k_{i+j-1}$ for some $1\leq j\leq s+1$. However, when $\mu=\tilde{\lambda}k_{i+j-1}$, consider the other columns of $(\xi_i,\cdots,\xi_{i+s})^T$, if each of them is linearly dependent with the columns of $M_{\mu}$, since $(\xi_i,\cdots,\xi_{i+s})^T1_n=0$, the $(j-1)$-th column is also linearly dependent with the columns of $M_{\mu}$, which makes a contradiction. Therefore, for any $\mu=\tilde{\lambda}k_{i+j-1}$, there exists another column (denoted to be the $i_0$-th column) in $(\xi_i,\cdots,\xi_{i+s})^T$ such that when the $(j-1)$-th column of $M$ is replaced by this column, $M$ is also invertible. Selecting $k_{i_0}\neq k_{i+j-1}$ makes $\bar{Q}$ of full row rank. This implies that for any selection of $k_1\neq k_2\neq\cdots\neq k_n\neq0$, the $i$-th, $\cdots,$ $(i+s)$-th rows of $\bar{Q}$ are linearly independent for all $\mu$. \\
Similarly, it is straightforward to prove that for different eigenvalues $\lambda_{i_1},\cdots,\lambda_{i_s}$, for any selection of $k_1, k_2,\cdots, k_n\neq0$ satisfying $\lambda_p k_i\neq\lambda_q k_j,~i,j=1,2,\cdots,n,~i\neq j,~p,q\in\{i_1,\cdots,i_s\}$, the $i_1,\cdots,i_s$-th rows of $\bar{Q}$ are linearly independent for all $\mu$. Furthermore, one can prove that for any $\Lambda$, almost all selections of $k_1, k_2,\cdots, k_n\neq0$
make $\bar{Q}$ be always of full row rank for all $\mu$. Since $\xi_1^TB\neq0$, one concludes that there exists a $K=\diag(k_1,\cdots,k_n)$ such that $rank(\mu I-LK,B)=n$ holds for all $\mu\in\mathbb{C}$, i.e., $(LK,B)$ is controllable.
\end{proof}

\begin{example}
Consider an interconnection topology depicted in Figure \ref{lemma}, if agent $1$ is the single leader, $(L,e_1)$ is obviously uncontrollable. However, let $K=\diag(1,1,2,1)$, one can see that $(LK,e_1)$ is controllable.
\end{example}
\begin{figure}
  \centering
  \includegraphics[width=1.6in]{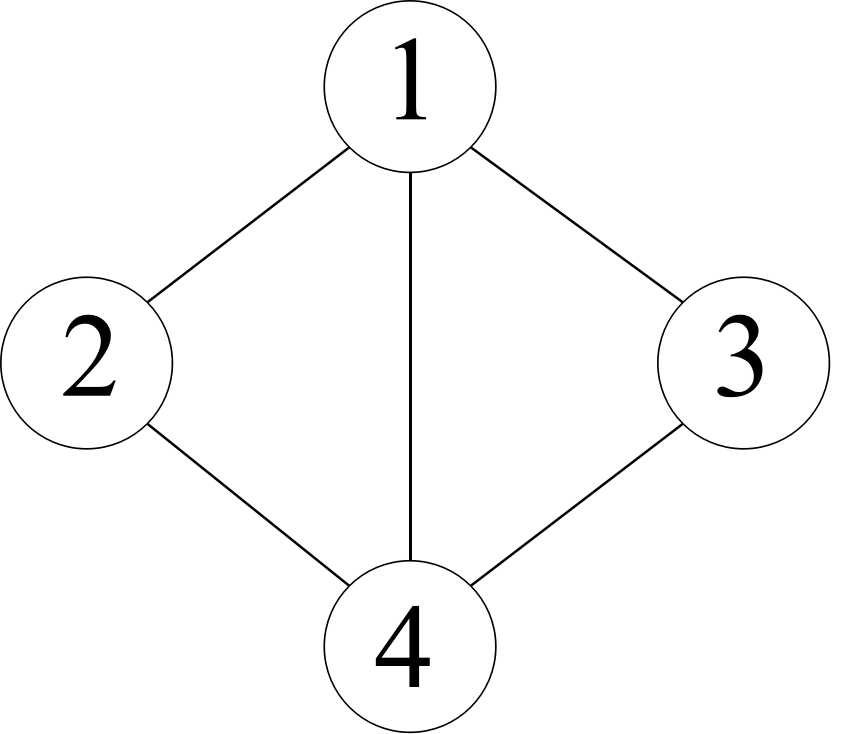}\\
  \caption{A graph with 4 nodes.}\label{lemma}
\end{figure}

\begin{corollary}
If the interconnection topology of an MAN is connected, then almost all selections of $K=\diag(k_1,\cdots,k_n)$ make $(LK,B)$ be controllable for any leader(s).
\end{corollary}
\begin{proof}
According to the proof of Lemma \ref{singleK}, if the interconnection topology is connected, for any single leader, almost all selections of $K$ could make $(LK,B)$ be controllable. Since the number of leaders is finite, it can be also obtained that almost all selections of $K$ could make $(LK,B)$ be controllable for any selection of a single leader. Since adding more leaders will never break the controllability, one concludes that almost all selections of $K=\diag(k_1,\cdots,k_n)$ guarantee that for any selection of leaders, $(LK,B)$ is controllable.
\end{proof}

\begin{lemma}\label{multipleK}
If the interconnection topology of an MAN is leader-follower connected, then there exists a $K=diag(k_1,\cdots,k_n)$ such that $(LK,B)$ is controllable.
\end{lemma}
\begin{proof}
If the interconnection topology is connected, this is exactly Lemma \ref{singleK}. If the interconnection topology contains $r>1$ connected components, then one can properly reorder the identifiers of the agents such that $L=\diag(L_1,\cdots,L_r),~B=\diag(B_1,\cdots,B_r)$. Refer to the blocks of $L$, divide $K$ into blocks $K=diag(K_1,\cdots,K_r)$, one obtains that $(LK,B)$ is controllable if and only if $(\diag(L_1K_1,\cdots,L_rK_r),\diag(B_1,\cdots,$ $B_r))$ is controllable, which holds if and only if $(L_iK_i,B_i),i=1,2,\cdots,r$ are controllable. However, according to Lemma \ref{singleK}, since $L_i$ corresponds to a connected subgraph and $B_i\neq0$, there exists a $K_i$ such that $(L_iK_i,B_i)$ is controllable, $i=1,2,\cdots,r$. Selecting $K_i$ making $(L_iK_i,B_i)$ be controllable ensures that $(LK,B)$ is controllable.
\end{proof}

\begin{lemma}\label{heterlemma}
If $(A_i,b_i),~i=1,\cdots,n$ are controllable, then Network (\ref{hetercompact}) is controllable if and only if the interconnection topology is leader-follower connected.
\end{lemma}
\begin{proof}
Suppose that $\Omega-\tilde{L}$ has an eigenvalue $\lambda$ with the corresponding left eigenvector $\theta^T=[\theta_1^T,\theta_2^T,\cdots,$ $\theta_m^T]$, where $\theta_i=(\theta_{i1},\theta_{i2},\cdots,\theta_{i m_i-1},\theta_{im_i})^T\in \mathbb{R}^{m_i}, i=1,2,\cdots,m$. Specially, denote $\theta_{im_i}=\eta_i$ and $\eta=(\eta_1,\cdots,\eta_n)^T$, one obtains that $\theta^T(\Omega-\tilde{L})=\lambda\theta^T$. Furthermore, $\lambda\theta_{i,m_i-j}=\theta_{i,m_i-j-1}+\eta^T\beta_{i,m_i-j}l_i$ yields
\begin{equation*}
  \left\{ {\begin{array}{*{20}c}
\begin{aligned}
    \theta_{i,m_i-1}&= \lambda \eta_i-\eta^T\beta_{i,m_i}l_i   \\
    \theta_{i,m_i-2}&= \lambda^2 \eta_i - \lambda\eta^T\beta_{i,m_i}l_i^T-\eta^T\beta_{i,m_i-1}l_i^T   \\
    &\vdots   \\
   \theta_{i,1}&= \lambda^{m_i-1} \eta_i - \lambda^{m_i-2}\eta^T\beta_{i,m_i}l_i^T-\cdots-\eta^T\beta_{i,2}l_i^T   \\
   \lambda\theta_{i,1}&= \eta^T \beta _{11}l_1  .
\end{aligned}
\end{array}} \right.
\end{equation*}
Therefore, $\eta^T(\rho_{1}l_1,\rho_{2}l_2,\cdots,\rho_{n}l_n)=\lambda^{m_j}\eta^T,~j=1,2,\cdots,n$, where $\rho_i=\beta_{i1}+\lambda\beta_{i2}+\cdots+\lambda^{m_i-1}\beta_{im_i},~i=1,2,\cdots,n$. This means Network (\ref{hetercompact}) is controllable if and only if there exist $\rho_1,\cdots,\rho_n$ such that $(L\cdot \diag(\rho_1,\cdots,\rho_n),B)$ is controllable. \\
Sufficiency: If the interconnection topology of Network (\ref{hetercompact}) is leader-follower connected, by Lemma \ref{multipleK}, there exists a $K$ such that $(LK,B)$ is controllable. Selecting $\beta_{i1}=k_i$ and $\beta_{ij}=0,~i=1,\cdots,n,~j=2,\cdots,m_i$ makes $(L\cdot \diag(\rho_1,\cdots,\rho_n),B)$ be controllable.\\
Necessity: If there exist $\rho_1,\cdots,\rho_n$ such that $(L\cdot \diag(\rho_1,\cdots,\rho_n),B)$ is controllable, it concludes that $L$ corresponds to a leader-follower connected graph. Otherwise, $L=\diag(L_1,\cdots,L_r)$, $L_i\in\mathbb{R}^{n_i},~i=1,2,\cdots,r,~n_1+\cdots+n_r=n$. Without loss of generality, suppose that the $r$-th connected component of the interconnection topology contains no leader, i.e., the last $r$ rows of $B$ are all $0$. Apparently, for any $K$, $LK$ contains a left eigenvector $\zeta^T=(0_{n_1}^T,\cdots,0_{n_{r-1}}^T,\tilde{\zeta}^T)$, which satisfies $\zeta^T B=0$. This contradicts that there exist $\rho_1,\cdots,\rho_n$ such that $(L\cdot \diag(\rho_1,\cdots,\rho_n),B)$ is controllable.
%
\end{proof}

\begin{theorem}\label{heterthm}
Network (\ref{hetercompact}) is controllable if and only if the following two conditions are satisfied simultaneously:\\
i) Matrix pairs $(A_i,b_i), ~i=1,2,\cdots,n$ are all controllable, respectively;\\
ii) The interconnection topology of the network is leader-follower connected.
\end{theorem}
\begin{proof}
Sufficiency: Since matrix pairs $(A_i,b_i), ~i=1,2,\cdots,n$ are all controllable, by Lemma \ref{heterlemma}, the interconnection topology being leader-follower connected makes network (\ref{hetercompact}) controllable.\\
Necessity: If $(A_i,b_i)$ is not controllable for some $i$, the state of agent $i$ is uncontrollable, which makes the whole network uncontrollable. Therefore, $(A_i,b_i),~i=1,\cdots,n$ must be controllable. According to Lemma \ref{heterlemma}, since Network (\ref{hetercompact}) is controllable, the interconnection topology must be leader-follower connected.
\end{proof}

Theorem \ref{heterthm} shows that Network (\ref{hetercompact}) is controllable if and only if the agents' dynamics are all controllable and the interconnection topology is leader-follower connected. This means Network (\ref{hetercompact}) can be controllable even if $(L,B)$ is not controllable. Especially, if $(L,B)$ is controllable, one can design the feedback gains as follows.

\begin{corollary}\label{hetercoro}
Suppose that $(A_i,b_i),~i=1,\cdots,n$ are controllable. If $(L,B)$ is controllable, then selecting $\beta_{11}=\beta_{21}=\cdots=\beta_{n1}\neq0$ and $\beta_{ij}=0,~i=1,2,\cdots,n,~j=2,3,\cdots,m_i$ makes Network (\ref{hetercompact}) controllable.
\end{corollary}
\begin{proof}
Refer to the proof of Lemma \ref{heterlemma}, if $\beta_{11}=\beta_{21}=\cdots=\beta_{n1}=q\neq0$ and $\beta_{ij}=0,~i=1,2,\cdots,n,~j=2,3,\cdots,m_i$, $\rho_1=\rho_2=\cdots=\rho_n=q$ and it holds that $\eta^TL=\frac{\lambda^{m_j}}{q}\eta^T,~j=1,2,\cdots,n$. This means that when $\beta_{11}=\beta_{21}=\cdots=\beta_{n1}\neq0$, the controllability of Network (\ref{hetercompact}) and $(L,B)$ are equivalent. Therefore, if $(A_i,b_i),~i=1,\cdots,n$ are supposed to be controllable, Network (\ref{hetercompact}) is also controllable under this selection of feedback gains.
\end{proof}

The following corollary has appeared in \cite{Wang09}, which can be directly obtained from Theorem \ref{heterthm} as a special case.
\begin{corollary}\label{kWNF}
For Network (\ref{hetercompact}), when $A_1=A_2=\cdots=A_n\triangleq A$, and $b_1=b_2=\cdots=b_n\triangleq b$, \\
1) The network is controllable if and only if $(A,b)$ is controllable and the interconnection topology is leader-follower connected;\\
2) If it is required that $\beta_1=\cdots=\beta_n$, then the network is controllable if and only if $(A,b)$ is controllable and $(L,B)$ is controllable.
\end{corollary}
\begin{proof}
If $A_1=A_2=\cdots=A_n\triangleq A$, and $b_1=b_2=\cdots=b_n\triangleq b$, then $(A_i, b_i),~i=1,2,\cdots,n$ all being controllable is equivalent to $(A,b)$ being controllable. By Theorem \ref{heterthm}, the network is controllable if and only if $(A,b)$ is controllable and the interconnection topology is leader-follower connected; if $\beta_1=\cdots=\beta_n$, by Corollary \ref{hetercoro}, the network is controllable if and only if $(A,b)$ is controllable and $(L,B)$ is controllable.
\end{proof}

\begin{remark}
Counter-intuitively, for the simple case that the agents are all of the first-order dynamic, leader-follower connection is only a necessary condition, but it is both necessary and sufficient for the general case, i.e., agents of heterogenous generic-linear dynamics. The reason of this difference relies in the feedback gains $\beta_1,\cdots,\beta_n$. Actually, the effect of $\beta_1,\cdots,\beta_n$ is to machining the state information of each agent. Degenerate it to the simple case $\dot x=u$, if one gives each agent an independent feedback gain $k_i$ respectively, the protocol turns to be $u_i=\sum\limits_{j\in N_i} a_{ij}(k_jx_j-k_ix_i)+u_o$, and the compact form is summarized as $\dot x=-LKx+Bu_o$. By Lemma \ref{multipleK}, the network (with first-order dynamic agents) is controllable (for almost all selections of $K$) if and only if the interconnection topology is leader-follower connected.
\end{remark}

To investigate controllability of heterogeneous MANs with uncontrollable nodes, controllability of MANs with inhomogeneous dynamics should be firstly discussed. Compare the two models in Networks (\ref{1234}) and (\ref{hetermodel}), if $m_1=m_2=\cdots=m_n=m$ and $\alpha_1=\alpha_2=\cdots=\alpha_n=0$ in (\ref{hetermodel}), the two models are exactly the same. Mathematically, the model of high-order dynamic agents with heterogeneous topologies is a special case of the model of heterogeneous dynamic agents. Therefore, the inhomogeneous dynamic is modeled for heterogeneous dynamic multi-agent networks. Consider the network
\begin{equation}\label{nonhomo}
\tilde x = (\Omega-\tilde{L})\tilde x + (e_i  \otimes \tilde b_i )u_o +f(t),
\end{equation}
where $f(t)=(f_1(t),\cdots,f_n(t))^T$, $f_i(t)=(f_{i1}(t),\cdots,f_{im_i}(t))^T$, and $f_{ij}(t)$ is piecewise continuous, $i=1,2,\cdots,n,~j=1,2,\cdots,m_i$. We compare its controllability with that of Network (\ref{hetercompact}).

\begin{theorem}\label{nonhomogeneous}
Controllability of Network (\ref{nonhomo}) is equivalent to that of Network (\ref{hetercompact}).
\end{theorem}
\begin{proof}
Consider the controllability of $\dot{x}=Ax+Bu+f(t), ~A\in\mathbb{R}^{n\times n}, B\in \mathbb{R}^{n\times m}$,  let $x=\tilde{x}+\bar{x}$ such that $\dot{\tilde{x}}=A\tilde{x}+Bu$ and $\dot{\bar{x}}=A\bar{x}+f(t)$. For any $t>t_0$, it holds that $\tilde{x}=e^{A(t-t_0)}\tilde{x}_0+\int_{t_0}^t e^{A(t-\tau)}Bu(\tau)d\tau$ and $\bar{x}=e^{A(t-t_0)}\bar{x}_0+\int_{t_0}^t e^{A(t-\tau)}f(\tau)d\tau$. According to the definition of controllability, the network is controllable if and only if for any $x^*\in\mathbb{R}^n$, there exists an input $u$ such that $\tilde{x}=x^*-\bar{x}$. However, it is equal to that for any $x^{**}\in\mathbb{R}^n$, there exists a control input $u$ such that $\tilde{x}=x^{**}$, i.e., $(A,B)$ is controllable. Therefore, for network (\ref{nonhomo}), it is controllable if and only if $(\Omega-\tilde{L},e_i  \otimes \tilde B_i )$ is controllable, which means $(L,B),(A_1,B_1),\cdots, (A_n,B_n)$ are all controllable. This implies that controllability of Network (\ref{hetercompact}) is equivalent to that of Network (\ref{nonhomo}).
\end{proof}

\begin{corollary}
Network (\ref{nonhomo}) is $\Xi$-uncontrollable if and only if Network (\ref{hetercompact}) is $\Xi$-uncontrollable, where $\Xi=\{\xi_1,\cdots,\xi_s\}$.
\end{corollary}
\begin{proof}
Network (\ref{nonhomo}) is controllable if and only if $(e_i  \otimes \tilde b_i,(\Omega-\tilde{L})e_i  \otimes \tilde b_i,\cdots,(\Omega-\tilde{L})^{m_1+\cdots+m_n}e_i  \otimes \tilde b_i)$ is of full row rank, which is a necessary and sufficient condition for Network (\ref{hetercompact}) to be controllable. This implies that the controllability matrices of Networks (\ref{nonhomo}) and (\ref{hetercompact}) are exactly the same.
\end{proof}

\begin{theorem}
For matrix pairs $(A_i,b_i)$, $i=1,2,\cdots,n$, if they are $\{\xi_{i1},\cdots,\xi_{is_i}\}$-uncontrollable, respectively, then Network (\ref{hetercompact}) is $\{\xi_{ij}\otimes e_i|i=1,\cdots,n;j=1,\cdots,s_i\}$-uncontrollable if and only if the interconnection topology is leader-follower connected.
\end{theorem}
\begin{proof}
Suppose that the controllability decomposition of $\dot{x}_i=A_ix_i+b_iu_i$ is
\begin{equation*}
  \left( {\begin{array}{*{20}c}
   {\dot {\bar {x} }_{ic} }  \\
   {\dot {\bar {x} }_{i \bar c} }  \\
\end{array}} \right) = \left( {\begin{array}{*{20}c}
   {\bar A_{ic} } & {\bar A_i }  \\
   0 & {\bar A_{i\bar c} }  \\
\end{array}} \right)\left( {\begin{array}{*{20}c}
   {\bar x_{ic} }  \\
   {\bar x_{i\bar c} }  \\
\end{array}} \right) + \left( {\begin{array}{*{20}c}
   {\bar b_{ic} }  \\
   0  \\
\end{array}} \right)u_i,
\end{equation*}
where $\bar{A}_{ic}=\left( {\begin{array}{*{20}c}
   0 & {I_{m_i - s_i - 1} }  \\
   -\alpha_{i0} & -\alpha_i^T  \\
\end{array}} \right), \tilde{b}_i=e_{m_i-s_i}$, i.e., $(A_i,b_i)$ is $\{\xi_{i1},\cdots,\xi_{is_i}\}$-uncontrollable, and there exists an invertible $Q_i\in\mathbb{R}^{m_i\times m_i}$ such that $Q_i(\xi_{i1},\cdots,\xi_{is_i})=(0,I_{s_i})^T$. Apparently, $\{\xi_{ij}\otimes e_i|i=1,\cdots,n;j=1,\cdots,s_i\}$ are linearly independent left eigenvectors of $A$ satisfying $\xi_i^Tb=0$, where $A=\diag(A_1,\cdots,A_n),$ $ b=\diag(b_1,\cdots,b_n)$ and $\diag(Q_1,\cdots,Q_n)\cdot \diag(\xi_{i1},\cdots,\xi_{i,s_1})=\diag(0,I_{s_i})$. For any left eigenvector of $A$, denoted as $\tilde\xi$, satisfying $\tilde\xi^Tb=0$, it holds that $\tilde\xi\in span\{\xi_{ij}\otimes e_i|i=1,\cdots,n;j=1,\cdots,s_i\}$, otherwise, $\diag(Q_1,\cdots,Q_n)\tilde\xi\notin span\{e_2(s_i),\cdots,e_{s_i}(s_i)\}$, which means there exists at least one $i$ such that $(A_i,b_i)$ is not $\{\xi_{i1},\cdots,\xi_{is_i}\}$-uncontrollable. With the protocol $u_i  =  - \alpha _i^T x_i  + \sum\limits_{j \in N_i } {a_{ij}  (\beta_j^T T_j x_j  - \beta_i^T T_i x_i )}  + u_{io}$, by Theorems \ref{heterthm} and \ref{nonhomogeneous}, one obtains that $\dot {\bar x}_c  = \bar A_c \bar x_c  + \bar A_{\bar c} \bar x_{\bar c}  + \bar b_c u$ is controllable if and only if $(\bar A_c,\bar b_c)$ is controllable and the interconnection topology is leader-follower connected, where $\dot {\bar x}_c=(\dot {\bar x}_{1c}, \dot {\bar x}_{2c}, \cdots, \dot {\bar x}_{nc})$, $\bar A_c=\diag(\bar A_{1c}, \bar A_{2c}, \cdots, \bar A_{nc})$, and $u=(u_1,u_2,\cdots,u_n)$. Considering that Network (\ref{hetercompact}) is $\{\xi_{ij}\otimes e_i|i=1,\cdots,n;j=1,\cdots,s_i\}$-uncontrollable if and only if  $\dot {\bar x}_c  = \bar A_c \bar x_c  + \bar A_{\bar c} \bar x_{\bar c}  + \bar b_c u$ is controllable under the protocol, and $(\bar A_c,\bar b_c)$ is controllable (which is proved), according to Definition \ref{xicontrollable}, one declares that Network (\ref{hetercompact}) is $\{\xi_{ij}\otimes e_i|i=1,\cdots,n;j=1,\cdots,s_i\}$-uncontrollable.
\end{proof}

\section{Examples}

In this section, two examples are provided to illustrate the main results in this paper.
\begin{example}\label{SNFE}
Consider Network (\ref{highcompact}) with five third-order agents, i.e., $\dot{x}_i^{(3)}=u_i,~i=1,2,3,4,5$. Suppose that the interconnection topologies of the first-order, second-order, third-order information are depicted in Figures \ref{examp11}, \ref{examp12} and \ref{examp13}, respectively. Obviously, if agent $1$ is the single leader, none of the graphs is controllable, even if Figure \ref{examp11} is leader-follower connected. However, the union graph of them is given in Figure \ref{examp10}, which is controllable. Actually, if one selects $k_1=k_2=k_3=1$, the high-order network becomes controllable.
\end{example}

\begin{figure}
\begin{minipage}[t]{0.5\linewidth}
\centering
\includegraphics[width=1.3in]{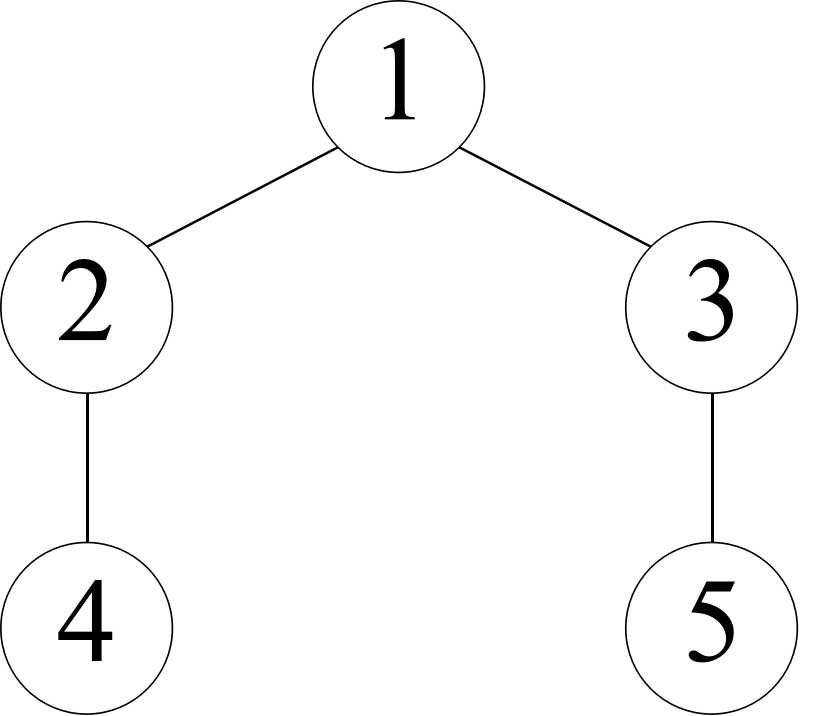}
\caption{The first-order information topology of system (\ref{highcompact}).}\label{examp11}
\end{minipage}%
~
\begin{minipage}[t]{0.5\linewidth}
\centering
\includegraphics[width=1.3in]{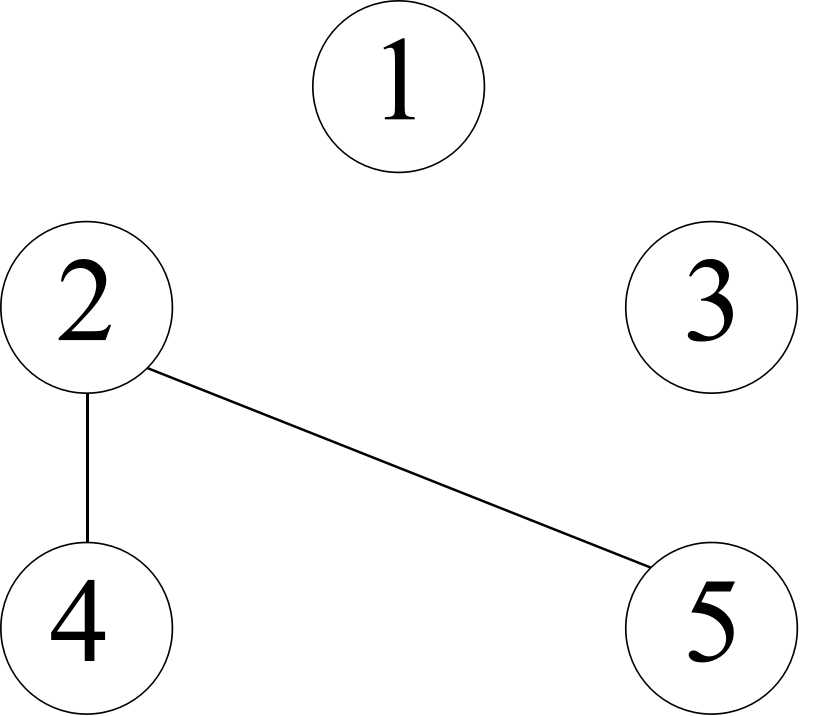}
\caption{The second-order information topology of system (\ref{highcompact}).}\label{examp12}
\end{minipage}
\end{figure}

\begin{figure}
\begin{minipage}[t]{0.5\linewidth}
\centering
\includegraphics[width=1.3in]{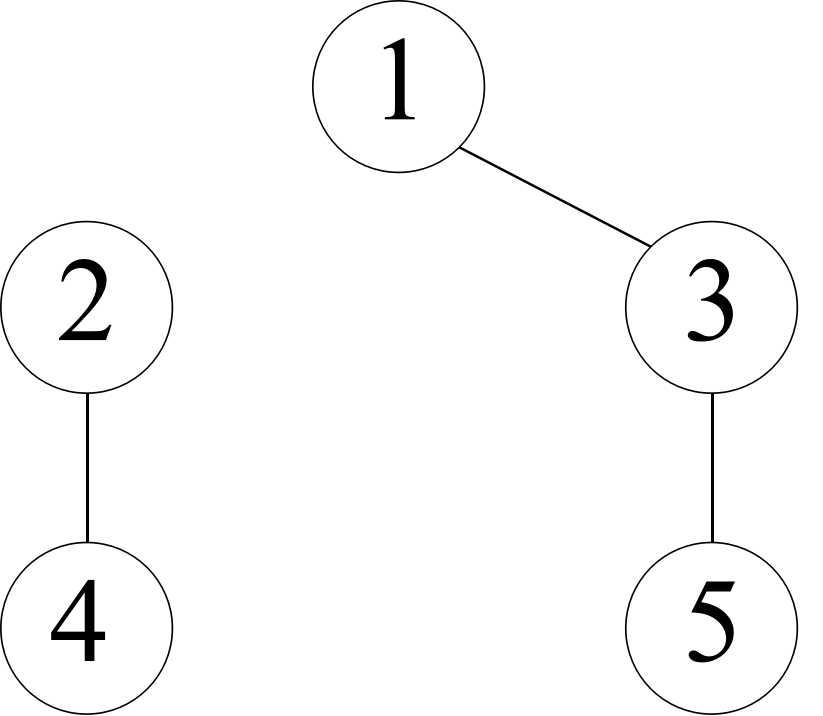}
\caption{The third-order information topology of system (\ref{highcompact}).}\label{examp13}
\end{minipage}%
~
\begin{minipage}[t]{0.5\linewidth}
\centering
\includegraphics[width=1.3in]{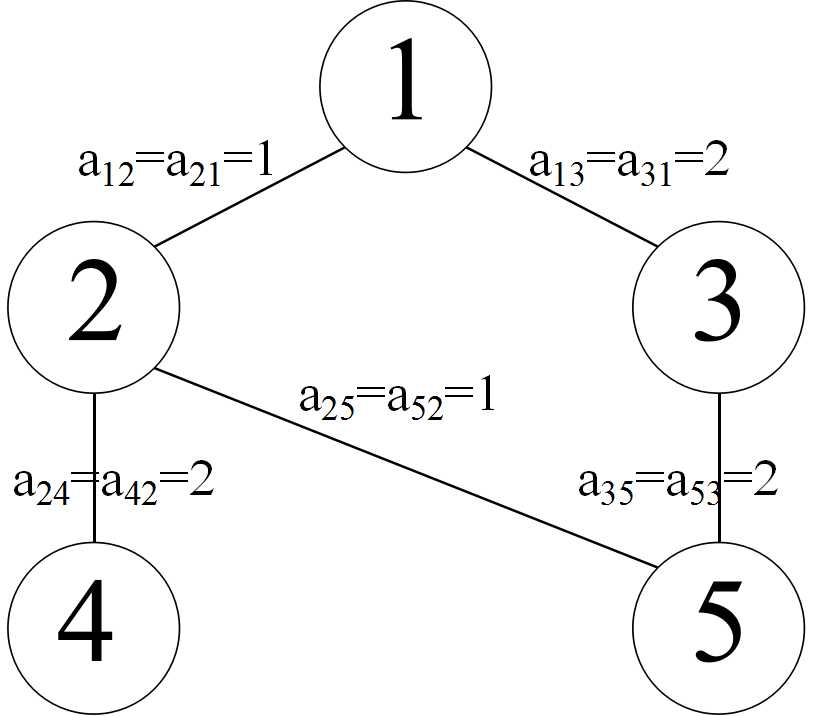}
\caption{The union graph of Figures 2, 3 and 4.}\label{examp10}
\end{minipage}
\end{figure}

\begin{example}\label{2}
For Network (\ref{hetercompact}) with heterogeneous dynamic agents, the interconnection topology is depicted in Figure \ref{examp2}. Suppose that the dynamics of the agents are $\dot{x}_1=u_1$, $\dot{x}_2=\left(
                              \begin{array}{cc}
                                1 & 1 \\
                                1 & 0 \\
                              \end{array}
                            \right)x_2+\left(\begin{array}{c}
                                         1 \\
                                         1
\end{array}\right)
u_1$, $\dot{x}_3=\left(
                   \begin{array}{ccc}
                     1 & 0 & 0 \\
                     1 & 1 & 0 \\
                     1 & 2 & 1 \\
                   \end{array}
                 \right)
x_3+\left(\begin{array}{c}
                                         1 \\
                                         0 \\
                                         1
\end{array}\right)
u_1$, respectively. Apparently, if agent $1$ is selected as the leader, the interconnection topology is leader-follower connected but $\left(\left(
                                                                                                                                          \begin{array}{ccc}
                                                                                                                                            2 & -1 & -1 \\
                                                                                                                                            -1 & 1 & 0 \\
                                                                                                                                            -1 & 0 & 1 \\
                                                                                                                                          \end{array}
                                                                                                                                        \right)
,\left(
   \begin{array}{c}
     1 \\
     0 \\
     0 \\
   \end{array}
 \right)
\right)$ is not controllable. However, if the feedback gains are selected as $\beta_1=1,~\beta_2=(2,0)^T,~\beta_3=(3,0,0)^T$, the network becomes controllable. Therefore, the effect of $\beta_1,\beta_2,\beta_3$ has two aspects. One is to make the agents able to communicate with each other, and the other is actually to benefit achieving controllability. In addition, if any of the two edges is broken, the network becomes uncontrollable immediately.
\end{example}
\begin{figure}
  \centering
  \includegraphics[width=1.6in]{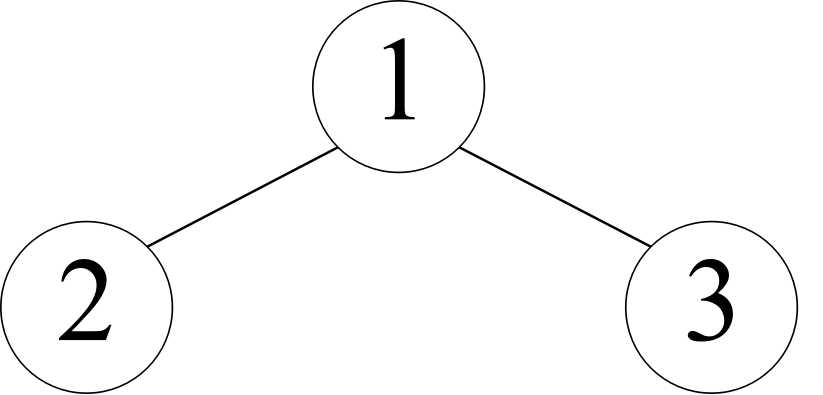}\\
  \caption{A graph with 3 nodes.}\label{examp2}
\end{figure}


\section{Conclusions}

In conclusion, we considered controllability of heterogeneous multi-agent networks. The main results in this paper provided graphic necessary and sufficient conditions on controllability of heterogeneous-topology networks and heterogeneous-dynamic networks. For an MAN with high-order dynamic agents, it is controllable if and only if there exists a Laplacian matrix, which is a linear combination of the Laplacian matrices of each order information, whose corresponding topology is controllable, and the topology corresponding to the first-order information should be leader-follower connected. For an MAN with different generic-linear dynamic agents, the necessary and sufficient condition for controllability of the network is that each agent contains a controllable dynamic and the interconnection topology of the network is leader-follower connected.
All the results in this paper are suitable for weighted topologies and multiple leaders.

\end{document}